\documentclass[a4paper,10pt]{article}
\usepackage[utf8]{inputenc}
\providecommand{\keywords}[1]{\textbf{\textit{keywords:}} #1}

\usepackage{amsthm}
\usepackage{amsmath}
\usepackage{amsfonts}
\usepackage{amssymb}
\usepackage{epsfig,epstopdf,titling}
\usepackage{caption}
\usepackage{subcaption}
\usepackage{multirow}
\usepackage{url}
\usepackage{graphicx}
\usepackage{pgfplots}
\usepackage{array,tabularx}
\usepackage{multicol}  
\usepackage{algorithmic} 
\usepackage[boxruled,longend,linesnumbered,ruled]{algorithm2e}
\usepackage{subcaption}
\theoremstyle{definition}

\theoremstyle{remark}

\newtheorem{theorem}{Theorem}
\usepackage{authblk}

\title{Tangential interpolatory projections for a class of second-order index-1 descriptor systems and application to Mechatronics}

\author[1]{Md. Motlubar Rahman}
\author[,2]{M. Monir Uddin\thanks{Corresponding author, monir.uddin@northsouth.edu}}
\author[3]{L. S. Andallah}
\author[4]{Mahtab Uddin}

\affil[1,3]{Department of Mathematics, Jahangirnagar University, Savar, Dhaka-1342, Bangladesh}
\affil[2]{Department of Mathematics and Physics, North south University, Dhaka-1229, Bangladesh}
\affil[4]{Institute of Natural Sciences, United International University, Dhaka-1212, Bangladesh}

\date{}

\begin{document}

\maketitle

\begin{abstract}
This paper studies the model order reduction of second-order index-1 descriptor systems using a tangential interpolation projection method based on the Iterative Rational Krylov Algorithm (IRKA). Our primary focus is to reduce the system into a second-order form so that the structure of the original system can be preserved. For this purpose, the IRKA based tangential interpolatory method is modified to deal with the second-order structure of the underlying descriptor system efficiently in an implicit way. The paper also shows that by exploiting the symmetric properties of the system the implementing computational costs can be reduced significantly.  Theoretical results are verified for the model reduction of the piezo actuator based adaptive spindle support which is second-order index-1 differential-algebraic form. The efficiency and accuracy of the method are demonstrated by analyzing the numerical results. 
\end{abstract}

\begin{center}
\end{center}

\section{Introduction} \label{sec:introduction}
We discuss the Iterative Rational Krylov Algorithm (IRKA) based 
tangential interpolation projection technique for the model reduction of 
second-order differential algebraic equations (DAEs) together with output 
equation which are given by
\begin{subequations} \label{eq:int:systemequation}
\begin{align}
     M_{11}\ddot{v}(t)+ L_{11} \dot{v}(t)+K_{11}v(t)+K_{12} \eta(t) & = F_1 u(t),\label{eq:int:systemequation1} \\
     K_{21}v(t)+K_{22} \eta(t)& = F_2 u(t), \label{eq:int:systemequation2}\\
     H_1 v(t)+ H_2 \eta(t)+ D_a u(t)& = y(t) ,\label{eq:int:systemequation3}
\end{align}        
\end{subequations} 
where $v(t)\in \mathbb{R}^{n_1}$, $\eta(t)\in\mathbb{R}^{n_2}$ are the states,
$u(t)\in \mathbb{R}^{m}$ are the inputs and $y(t)\in \mathbb{R}^{p}$ are the outputs, and matrices $M_{11}, L_{11},  K_{11}, K_{12}, K_{21}$ and $K_{22}$ are  sparse. The matrix $D_a\in\mathbb{R}^{p\times m}$ represents the direct feed-through from the input to the output. We consider that number of inputs and outputs is greater than one i.e., the system is multi-inputs and multi-outputs (MIMO). We also assume that the block matrix $K_{22}$ is non-singular. In the previous literature see, e.g., \cite{morUddSKetal12} such system was defined as index-1 system. This system is called symmetric if the matrices $M_{11}$, $L_{11}$, $K_{11}$, $K_{22}$ and $D_a$ are symmetric, and $K_{21}=K_{12}^T$, $H_1= F_1^T$ and $H_2= F_2^T$.

Such structure systems arise in many applications, for examples in the modeling of the mechanical and electrical networks (see e.g., \cite{EicF98,fuchs2011power,neugebauer2010control}) where the constraints are imposed to control the dynamic behavior of the systems or mechatronics \cite{zaeh2011prediction,NeuDW07} in which mechanical and electrical components are coupled with each other. In the specific case of the model example which is used for our numerical experiments is mechatronics, the index-1 character results from the certain machine tools; Adaptive Spindle Support (ASS) \cite{drossel2008,neugebauer2010} based on piezo actuators. 
\begin{figure}[!htbp]
\begin{center}
\includegraphics[width=100mm,height=60mm]{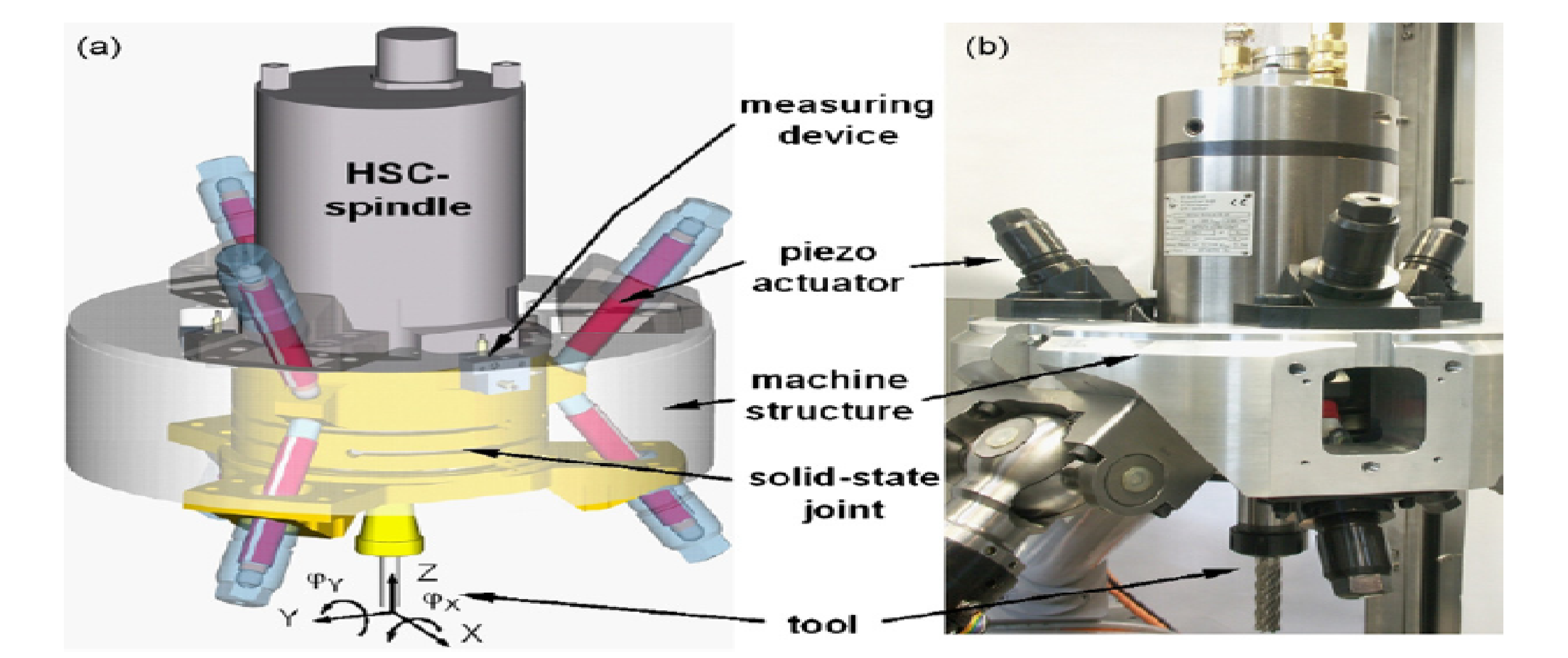}
\caption{(a) Piezo-actuator based Adaptive Spindle Support (ASS) and (b) real component mounted on the test bench 3pod (Source \cite{drossel2008}).}
\label{fig:pizomechsystem}
\end{center}
\end{figure}
Piezo-actuator-based ASS as shown in  Figure~\ref{fig:pizomechsystem}(a) is a machine tool that is mounted in a parallel kinematic machine shown in Figure~\ref{fig:pizomechsystem}(b), to attain the additional positioning freedom during machining operations. The detail of such a complex mechatronic model can be found, for example in \cite{drossel2008,neugebauer2010} for more detail.  The purpose of the piezo-sensor and piezo-actuator is to control active vibration or shunt damping so that a high-quality product can ensure that is an indispensable characteristic of production engineering in the commercial sense.  For analyzing the mechanical design and performance of the ASS, using the Finite Element Method (FEM) a mathematical model as defined in (\ref{eq:int:systemequation}) was formed, where the time-dependent state vector $v(t)$ consists of the components of mechanical displacements, $\eta(t)$ is the electrical charges, and $M_{11}$, $L_{11}$, and $K_{11}$ are the mass, damping, and stiffness, respectively. Moreover, the block  $K_{22}$ is electrical, and $K_{21}=K_{12}^{T}$ are coupling terms, the general force quantities (mechanical forces, and electrical charges) are chosen as the input quantities $u$, and the corresponding general displacements (mechanical displacements and electrical potential) are the output quantities $y$. The total mass matrix contains zeros at the locations of electrical potential. More precisely, the electrical potential of piezo-mechanical systems (Degrees of Freedom (DoF) for the electrical part) is not associated with inertia. The equation of motion of the mechanical system in (\ref{eq:int:systemequation}) can be found in \cite{neugebauer2010making}. This equation results from a finite element discretization of the balance equations. For piezo-mechanical systems, these are the mechanical balance of momentum (with inertia term) and the electro-static balance. From this, the electrical potential without inertia term is obtained. Thus, for the whole system (mechanical and electrical DoF) the mass matrix has a rank deficiency.
 
If the model is too large, performing the simulation with it has prohibitively expensive computational effort, or is simply impossible due to the limited  computer memory. Therefore, we want to replace a large-scale system with a substantially small-scale system that approximates the main features of the original systems but is much faster to evaluate. Model Order Reduction (MOR) has a vital impact on scientific research to accumulate the computational approaches to the engineering fields, especially, industrial applications. The time-management is the key feature in modern production systems and MOR techniques can play an essential role by reducing computational costs. Numerical techniques via computer simulation can act as a bridge between the factory production and scientific research.

Model Order Reduction (MOR) of the index-1 descriptor system (\ref{eq:int:systemequation}) was studied in several literature, see, e.g., \cite{morUddSKetal12,morUdd11,morBenSU12,morUdd19}. All these literature focused on the system theoretic method Balanced Truncation (BT) considering either  second-order to first-order or second-order to second-order reduction techniques. To implement the method one has to compute and store the Gramian factors of the system. Computing the Gramian factors by solving continuous-time algebraic Lyapunov equations is a huge computational task and is often considered as a drawback of the method.

On the other hand, Iterative Rational Krylov Algorithm (IRKA) based interpolatory methods as introduced in \cite{morGugAB08,morAntBG10} is computationally efficient. Therefore, recently this method is applied frequently for the model reduction of large-scale dynamical systems. The method was generalized for first-order descriptor system in \cite{morGugSW13}. The idea was also extended in \cite{morBeaG09,morWya12,motlubarspmor2020} for the second-order to second-order reduction of second-order standard systems. Authors in \cite{benner2016structure} discussed this method to obtain reduced first-order state-space model from the second-order index-1 system in (\ref{eq:int:systemequation}). Until now there is no investigation of this method for the second-order to second-order model reduction of second-order index-1 systems. This paper is mainly devoted to close this gap.

In this  paper, we will discuss the Structure-Preserving Model Order Reduction (SPMOR), i.e., second-order to second-order model order reduction of the second-order index-1 descriptor systems applying tangential interpolation projection based on IRKA. Generally, the second-order index-1 system (\ref{eq:int:systemequation}) can be converted into a second-order standard system. Then the proposed method can be applied to the converted system following the procedure as in \cite{morWya12}. However, such conversion will destroy the sparsity pattern and turn the system into a dense form. The dense system not only consumes a large-scale computer memory but also leads to additional computational complexities. For a large-scale system, like the Adaptive Spindle Support (ASS) model consider in this paper, converting into dense form is forbidden. We develop the SPMOR algorithm for the  system (\ref{eq:int:systemequation}) without converting the system dense form explicitly. For this purpose, the standard IRKA based interpolatory methods as in \cite{morBeaG09} would be modified  to deal with the second-order structure. In many cases, in real-life applications, see, e.g. \cite{benner2013improved,wagner2003symmetric,chahlaoui2006model}, the model we use in the numerical experiments, systems are in symmetric form. This paper also shows how to accelerate the computation by exploiting the symmetric properties of the system. The proposed techniques are applied to the large-scale real-life model, piezo actuator based adaptive spindle support. The efficiency of the method is discussed  by the numerical results. The results are also compared with that of the Balanced Truncation. Note that for the balancing-based model reduction we have considered the procedures exactly presented in \cite{benner2016structure} and to avoid the elaborations we have not discussed this in the paper again.    

\section{IRKA based tangential interpolatory methods}\label{sec:standard}
The goal of this section is to review the basic idea of the tangential interpolation techniques based on IRKA from the previous literature. At first, we introduce the method for the first-order generalized systems. Then the idea would be generalized for the second-order standard systems. This section also recalls some important definitions and essential notations, theorems, etc.,  that will be used in the next sections.   

\subsection{Tangential interpolation for first-order systems}\label{subsec:1st_order_standard}
We briefly discuss the IRKA based tangential interpolation method for the MIMO generalized state-space system
\begin{equation}\label{eq:standard:ltisystem}
\begin{aligned}
E\dot x(t) &= A x(t)+B u(t),\\
y(t) &= C x(t)+ D_a u(t),
\end{aligned}
\end{equation}
where $ E \in \mathbb{R}^{k\times k}$ is non-singular, 
and $ A\in\mathbb R ^{k\times k}$, $B\in\mathbb R^{k\times p}$,
$C\in\mathbb R ^{m\times k}$ and ${D}_a\in\mathbb R^{m\times p}$. 
The transfer function of the system (\ref{eq:standard:ltisystem}) is defined by $\text{G}(s)= C (sE-A)^{-1}B + D_a$, where $s\in \mathbb{C}$. Applying the tangential interpolatory framework we want to construct an $r$-dimensional ($r\ll k$) reduced-order model
\begin{equation}
\label{eq:standard:reducedltisystem}
\begin{aligned}
\hat{E}\dot {\hat{x}}(t) &= \hat{A} \hat{x}(t)+\hat{B} u(t),\\ 
\hat{y}(t) &=\hat{C}\hat{x}(t)+ \hat{D}_a u(t),
\end{aligned}
\end{equation}
such that its transfer 
function $\hat{\text{G}}(s) = \hat{C} (s\hat{E}-\hat{A})^{-1}\hat{B} + \hat{D}_a$ 
interpolates the original one, $\text{G}(s)$, at selected
points in the complex plane along with selected directions. The
points are called interpolation points and the directions are called tangential directions. We use the procedure illustrated in \cite{morBeaG09} to make this problem more precisely as follows. 

Initially, we consider a set of ad-hoc interpolation points $\left\{\alpha_i\right\}_{i=1}^{r}$, right tangential directions $\left\{b_i\right\}_{i=1}^{r}$ and left tangential directions $\left\{c_i\right\}_{i=1}^{r}$ to construct two $n\times r$ projection matrices 
\begin{equation}\label{eq:standard:solution1}
\begin{aligned}
V & = \left[(\alpha_1 E-A)^{-1}Bb_1, \cdots,(\alpha_r E-A)^{-1}Bb_r\right],\\
W & = \left[(\alpha_1E-A)^{-T}C^Tc_1, \cdots,(\alpha_rE-A)^{-T}C^Tc_r\right].
\end{aligned}
\end{equation}
Then the interpolation points and those tangential directions need to be updated until the reduced transfer function interpolates the original transfer function reasonably. Since the continuous updates of the interpolation points gradually match the eigenvalues of the system, the initial ad-hoc consideration will not affect the convergence of the approach.

Now, approximating $x(t)$ by $V\hat{x}(t)$ and enforcing the Petrov-Galerkin condition provided as
$$
W^T(EV\dot{\hat{x}}(t)-AV\hat{x}(t)-Bu(t))=0, \quad \hat{y}(t)=CV\hat{x}(t)+D_a u(t),
$$
construct the reduced matrices in (\ref{eq:standard:reducedltisystem})
as
\begin{equation}
\label{eq:standard:reducedmatrices}
\begin{aligned}
\hat{E}:= {W}^T  E V,\quad \hat{A}:= W^T A V, \quad
\hat{B}:= W ^T B , \quad \hat{C}:= C V, \quad \hat{ D}_a:=  D_a.
\end{aligned}
\end{equation}

The reduced  model obtained by this procedure satisfies
\begin{align}\label{eq:standard:hermitcond}
\text{G}(\alpha_i)b_i=\hat{\text{G}}(\alpha_i)b_i,\,
c_i^T\text{G}(\alpha_i)b_i=c_i^T\hat{\text{G}}(\alpha_i)b_i \,\, \text{and}\,\,
c_i^T\text{G}'(\alpha_i)b_i=c_i^T\hat{\text{G}}'(\alpha_i)b_i,
\end{align}
for $i = 1,2, \dots , r,$ which is known as Hermite bi-tangential interpolation conditions.

The quality of the reduced-order model (ROM) can be measured by $|y-\hat{y}|$,
which can, in frequency domain, also be expressed in terms of the transfer function error
\begin{align}
\| \text{G}(.)-\hat{\text{G}}(.)\|.
\end{align}

Common choices for the error norm are the $\mathcal{H}_\infty$ or $\mathcal{H}_2$-norms (see, e.g. \cite{morAnt05}).
To  minimize the error, the choice of interpolation points and tangential directions are crucial tasks. They depend on the reduced-order model; hence are not known priory. The Iterative Rational Krylov Algorithm (IRKA) introduced in \cite{morGugAB08} resolves the problem by iteratively correcting the interpolation points and the directions as summarized in Algorithm~\ref{alg:irka1}.
\begin{algorithm}[t]
	\SetAlgoLined
	\SetKwInOut{Input}{Input}
	\SetKwInOut{Output}{Output}
	\caption{IRKA for First-Order MIMO Systems.}
	\label{alg:irka1}
	\Input {$E, A, B, C, D_a$.}
	\Output{$\hat{E}, \hat{A}, \hat{B}, \hat{C}$, $\hat{D}_a:= D_a$.}
	Make the initial selection of the interpolation points $\{\alpha_i\}_{i=1}^r$ and the tangential directions $\{b_i\}_{i=1}^r$ and $\{c_i\}_{i=1}^r$.\\
	Construct
	\begin{align*}
	&V = \begin{bmatrix}(\alpha_1E-A)^{-1}Bb_1,\cdots, 
	(\alpha_r E-A)^{-1}Bb_r \end{bmatrix}, \\
	&W = \begin{bmatrix}(\alpha_1 E^T-A^T)^{-1}C^T c_1,\cdots ,
	(\alpha_r E^T-A^T)^{-1}C^T c_r \end{bmatrix}.
	\end{align*}\\
	\While{(not converged)}{%
	
	Compute $\hat{E} = W^T E V$, $\hat{A} = W^T A V$, $\hat{B} = W^T B$ and $\hat{C} = C V$.\\
	Compute $\hat{A}z_i = \lambda_i\hat{E}z_i$ and $y^*_i \hat{A} = \lambda_i y^*_i \hat{E}$ for $\alpha_i \leftarrow -\lambda_i$, $b^*_i \leftarrow -y^*_i \hat{B}$ and $c^*_i \leftarrow \hat{C}z^*_i$, for $i=1,\cdots ,r$.\\
	Repeat step~2. \\
	}
	Construct the reduced-order matrices
	$\hat{E} = W^T E V, \hat{A} = W^T A V, \hat{B} = W^TB$ and $\hat{C} = C V$.\\
\end{algorithm}
\subsection{Tangential interpolation for standard second-order systems}\label{subsec:itmsecstd}
Let us move to review of second-order linear time-invariant (LTI) continuous-time system 
\begin{align}\label{eq:tisec:systemequation}
     M\ddot{z}(t)+ L \dot{z}(t)+Kz(t) & = F u(t),\quad 
     y(t) = H z(t)+D_au(t),
\end{align}          
where $M, L$ and $K$ are non-singular, and $z(t)$ is the $n$ dimensional 
state vector. Consider that the system is MIMO, and its transfer function can be defined as
\begin{align}\label{eq:tisec:stsectf}
 \tilde{G}(s)= H (s^2M+sL+K)^{-1}F + D_a; \quad s\in \mathbb{C}.
\end{align}

Transform the system (\ref{eq:tisec:systemequation}) into an equivalent first-order form (\ref{eq:standard:ltisystem}), in which 
$x(t) = \left[ {\dot{z}(t)}^T \, {z(t)}^T \right]^T$ and the coefficient matrices are replaced by 
\begin{equation}\label{eq:tisec:secondchoice}
\tilde{E}:= \underbrace{\begin{bmatrix}
	
   0 & M\\
   M & L
   \end{bmatrix}}_{E}, 
\tilde{A}:= \underbrace{\begin{bmatrix}
      M &  0\\
       0 &  -K
    \end{bmatrix}}_{A},~
\tilde{B}:= \underbrace{\begin{bmatrix}
       0\\F
    \end{bmatrix}}_{B}, 
\tilde{C}:=
      \underbrace{\begin{bmatrix}
       0 & H
      \end{bmatrix}}_{C}~ \text{and}~ D_a= D_s.
\end{equation}

Although there are several first-order representations of the second-order system as shown in \cite{morSal05}, we are particularly interested in this form (\ref{eq:tisec:secondchoice}); since this representation yields first-order symmetric system if $M$, $L$, $K$ are symmetric, $F$ and $H$ are transposes of each other. Once the system in (\ref{eq:tisec:systemequation}) is converted into the system in (\ref{eq:tisec:secondchoice}), Algorithm~\ref{alg:irka1} can be applied to obtain a reduced-order model. However, the reduced-order model is in first-order form and one can not go back to the second-order representation since the second-order structure has already disintegrated. Therefore we aim to obtain an $r-$dimensional $(r\ll n)$ second-order reduced model
\begin{align}\label{eq:tisec:reducedsystem}
     \hat{M}\ddot{\hat{z}}(t)+ \hat{L} \dot{\hat{z}}(t)+\hat{K}\hat{z}(t) & = \hat{F} u(t),\quad 
     \hat{y}(t) = \hat{H} \hat{z}(t)+ \hat{D}_a u(t),
\end{align}
where using the projection matrices $V_s,W_s \in \mathbb{R}^{n\times r}$, the coefficient matrices 
are obtained as follows
\begin{equation}
\begin{aligned}
 \hat{M} &=W_s^T M V_s, \ \hat{L}= W_s^T L V_s, \ \hat{K}= W_s^T K V_s,\\ \hat{F} &=W_s^T F, \ \hat{H}=HV_s\ \text{and}\ \hat{D}_a:=D_a. 
\end{aligned}
\end{equation}

We want to achieve this by applying tangential interpolatory techniques. It can be shown that the transfer function of the second-order 
system (\ref{eq:tisec:systemequation})  coincides with the transfer function of its first-order representation in (\ref{eq:tisec:secondchoice}), i.e., 
\begin{align*}
 \tilde{G}(s)= H (s^2M+sL+K)^{-1}F+D_a = \tilde{C} (s\tilde{E}-\tilde{A})^{-1}\tilde{B}+D_s.
\end{align*}

Therefore, based on the discussion above the interpolatory projection method can directly be applied to (\ref{eq:tisec:systemequation}) for the reduced model in (\ref{eq:tisec:reducedsystem}).
Considering interpolation points 
$\left\{\alpha_i\right\}_{i=1}^{r}$, right tangential directions
$\left\{b_i\right\}_{i=1}^{r}$ and left tangential directions 
$\left\{c_i\right\}_{i=1}^{r}$, and construct $V_s$ and $W_s$ as follows
\begin{equation}\label{eq:standard:solution2}
\begin{aligned}
V_s & = \left[(\alpha_1^2 M+\alpha_1L+K)^{-1}Fb_1, \cdots,(\alpha_r^2 M+\alpha_rL+K)^{-1}Fb_r\right],\\
W_s & = \left[(\alpha_1^2 M+\alpha_1L+K)^{-T}H^Tc_1, \cdots,(\alpha_r^2 M+\alpha_rL+K)^{-T}H^Tc_r\right].
\end{aligned}
\end{equation}

If the reduced-order model (\ref{eq:tisec:reducedsystem}) is constructed by  $V_s$ and $W_s$,
the reduced transfer function $\hat{\tilde{\text{G}}}(s)= \hat{H} (s^2\hat{M}+s\hat{L}+\hat{K})^{-1}\hat{F}+\hat{D}_a$
tangentially interpolates $\tilde{\text{G}}(s)$ satisfying the  interpolation conditions as in (\ref{eq:standard:hermitcond}).
  
In some articles, see, e.g., \cite{morWya12,xu2018structure} the SPMOR of the second-order system via tangential interpolations were discussed from the first-order representations as in (\ref{eq:tisec:secondchoice}). There the authors discussed that how the interpolation points and the tangential directions for the second-order system can be efficiently updated by the corresponding first-order form using Algorithm~\ref{alg:irka1}.

Note that if the second-order system (\ref{eq:tisec:systemequation})) is symmetric the projection matrices $V_s$ and $W_s$  have coincided. In that case, we can reduce the computational cost to construct the reduced models. Another important issue for the SPMOR is to update the interpolation points and tangential directions. We leave this to discuss in the next section.

\section{SPMOR for second-order index-1 descriptor systems}\label{sec:spmor-secordIndex1}
In this section, our goal is to develop interpolatory projections for SPMOR 
of second-order index-1 DAEs (\ref{eq:int:systemequation}). In Section~\ref{sec:introduction}, we already have mentioned that second-order index-1 DAEs can be converted into a second-order standard system. In a large-scale system, this conversion is however infeasible. This section is mainly devoted without such converting how to apply the tangential interpolatory methods for the SPMOR of second-order DAEs.

\subsection{IRKA based sparse tangential interpolation} \label{sec:sparseIRKA}
Recall the second-order index-1 system (\ref{eq:int:systemequation}), and rewrite the system in Matrix-vector form:
\begin{subequations}
\label{eq:spmor-secordIndex1:matrixvecform}
\begin{align}
\underbrace{\begin{bmatrix} M_{11} & 0 \\0 & 0 \end{bmatrix}}_{\bar{M}} \begin{bmatrix} \ddot v (t) \\ \ddot \eta (t) \end{bmatrix} + 
\underbrace{\begin{bmatrix} L_{11} & 0 \\ 0 & 0 \end{bmatrix}}_{\bar{L}} \begin{bmatrix}\dot v (t) \\ \dot \eta (t) \end{bmatrix} &+ 
\underbrace{\begin{bmatrix} K_{11} & K_{12} \\ K_{21} & K_{22} \end{bmatrix}}_{\bar{K}} \begin{bmatrix} v (t) \\ \eta (t) \end{bmatrix} = 
\underbrace{\begin{bmatrix} F_1 \\ F_2 \end{bmatrix}}_{\bar{F}} u(t),\label{eq:spmor-secordIndex1:matrixvecform1}\\
y(t) &= \underbrace{\begin{bmatrix} H_1 & H_2 \end{bmatrix}}_{\bar{H}} \begin{bmatrix} v (t) \\ \eta (t) \end{bmatrix}+D_a u(t).\label{eq:spmor-secordIndex1:matrixvecform2}
\end{align}
\end{subequations}
The transfer function matrix of the system is defined by 
\begin{align}\label{eq:spmor-secordIndex1:tfindex1}
 \bar{G}(s)= \bar{H} (s^2\bar{M}+s\bar{L}+\bar{K})^{-1}\bar{F}+D_a. 
\end{align}
From second equation of  (\ref{eq:spmor-secordIndex1:matrixvecform1}) we obtain
$$ \eta(t) = -K_{22}^{-1}K_{21}v(t)+ K_{22}^{-1}F_2 u(t).$$
Inserting this identity into the first equation of (\ref{eq:spmor-secordIndex1:matrixvecform1}) 
and equation (\ref{eq:spmor-secordIndex1:matrixvecform2}), and some algebraic manipulations yield
\begin{align}\label{eq:spmor-secordIndex1:systemequation}
    \mathcal{M}\ddot{v}(t)+ \mathcal{L} \dot{v}(t)+\mathcal{K}v(t)  = \mathcal{F} u(t),\quad \text{and}\quad
     y(t) = \mathcal{H} v(t)+\mathcal{D}_au(t),
\end{align}
respectively, where
\begin{equation}  
\label{eq:spmor-secordIndex1:schurcomp}
\begin{aligned}
\mathcal{M}:& = M_{11}, \quad \mathcal{L}: = L_{11} \\
\mathcal{K}: &= K_{11} - K_{12}{K_{22}}^{-1}K_{21}, \quad  
\mathcal{F}: = F_1 - K_{12}{K_{22}}^{-1}F_2, \\
\mathcal{H}: &= H_1 - H_2{K_{22}}^{-1}K_{21},\quad
\mathcal{D}_a: = D_a + H_2{K_{22}}^{-1}F_2.
\end{aligned}
\end{equation}

This system is LTI  continuous-time system and can be compared with the standard second-order system as in (\ref{eq:tisec:systemequation}). The transfer function matrix of the system (\ref{eq:spmor-secordIndex1:systemequation}) is given by 
\begin{align}\label{eq:spmor-secordIndex1:tfstand}
\mathcal{G}(s)= \mathcal{H} (s^2\mathcal{M}+s\mathcal{L}+\mathcal{K})^{-1}\mathcal{F}+\mathcal{D}_a. 
\end{align}

The following observation shows that systems (\ref{eq:spmor-secordIndex1:matrixvecform}) and 
(\ref{eq:spmor-secordIndex1:systemequation}) are equivalent.

\begin{theorem}\label{eq:spmor-secordIndex1:thm}
The transfer-function matrices $\bar{G}(s)$ and $\mathcal{G}(s)$ as defined in 
(\ref{eq:spmor-secordIndex1:tfindex1}) and (\ref{eq:spmor-secordIndex1:tfstand}), respectively
are equal. 
\end{theorem}

\begin{proof}
Plugging $\bar{H}$, $\bar{M}$, $\bar{D}$, $\bar{K}$ and $\bar{L}$ from 
(\ref{eq:spmor-secordIndex1:matrixvecform}) into (\ref{eq:spmor-secordIndex1:tfindex1}) we obtain 
\begin{align}\label{eq:spmor-secordIndex1:thm1}
\bar{G}(s) & = \begin{bmatrix} H_1 & H_2 \end{bmatrix} \left(s^2\begin{bmatrix} M_{11} & 0 \\0 & 0 \end{bmatrix}+s\begin{bmatrix} L_{11} & 0 \\ 0 & 0 \end{bmatrix}+\begin{bmatrix} K_{11} & K_{12} \\ K_{21} & K_{22} \end{bmatrix}\right)^{-1}\begin{bmatrix} F_1 \\ F_2 \end{bmatrix}+D_a\nonumber\\ 
& = \begin{bmatrix} H_1 & H_2 \end{bmatrix} \begin{bmatrix}s^2 M_{11}+ sL_{11}+K_{11} & K_{12} \\K_{21} & K_{22} \end{bmatrix}^{-1}\begin{bmatrix} F_1 \\ F_2 \end{bmatrix}+D_a. 
\end{align}
 
Consider that 
\begin{align*}
\begin{bmatrix}s^2 M_{11}+ sL_{11}+K_{11} & K_{12} \\K_{21} & K_{22} \end{bmatrix}^{-1}\begin{bmatrix} F_1 \\ F_2 \end{bmatrix}&= \begin{bmatrix} \textbf{x}_1 \\ \textbf{x}_2 \end{bmatrix},
\end{align*}
which leads
\begin{align*}
\begin{bmatrix}s^2 M_{11}+ sL_{11}+K_{11} & K_{12} \\K_{21} & K_{22} \end{bmatrix}\begin{bmatrix}  \textbf{x}_1 \\ \textbf{x}_2 \end{bmatrix} = \begin{bmatrix} F_1 \\ F_2 \end{bmatrix}.
\end{align*} 

This implies
\begin{align}
(s^2 M_{11}+ sL_{11}+K_{11})\textbf{x}_1 +  K_{12}\textbf{x}_2 & = F_1, \label{eq:spmor-secordIndex1:lsys1}\\
K_{21}\textbf{x}_1+ K_{22} \textbf{x}_2 & =  F_2. \label{eq:spmor-secordIndex1:lsys2}
\end{align}

Equation (\ref{eq:spmor-secordIndex1:lsys2}) gives
\begin{align*}
\textbf{x}_2 = -K_{22}^{-1}K_{21}\textbf{x}_1 + K_{22}^{-1}F_2.
\end{align*}

Inserting this identity into equation
(\ref{eq:spmor-secordIndex1:lsys1}) we have
 \begin{align}
\textbf{x}_1 = (s^2 M_{11}+ sL_{11}+K_{11}-K_{12}K_{22}^{-1}K_{21})^{-1}(F_1-K_{12} K_{22}^{-1}F_2).
\end{align}

Equation (\ref{eq:spmor-secordIndex1:thm1}) implies
\begin{align*}
\bar{G}(s)  = \begin{bmatrix} H_1 & H_2 \end{bmatrix} \begin{bmatrix} \textbf{x}_1 \\ \textbf{x}_2 \end{bmatrix}+D_a=
  H_1\textbf{x}_1 + H_2\textbf{x}_2+D_a.
\end{align*}

Using $\textbf{x}_1$ and $\textbf{x}_2$, and some algebraic manipulations lead the above equation to the form 
\begin{align*}
\bar{G}(s)  = & (H_1 - H_2{K_{22}}^{-1}K_{21})(s^2 M_{11}+ sL_{11}+K_{11}-K_{12}K_{22}^{-1}K_{21})^{-1}\\ &\quad(F_1-K_{12} K_{22}^{-1}F_2)+ (D_a+H_2K_{22}^{-1}F_2).
\end{align*}

Now following (\ref{eq:spmor-secordIndex1:schurcomp}) we obtain
\begin{align*}
\bar{G}(s)  =  \mathcal{H}(s^2 \mathcal{M}+ s\mathcal{L}+\mathcal{K})^{-1}\mathcal{F}+ \mathcal{D}_a,
\end{align*}
which leads to the desired conclusion.
\end{proof}

We are now ready to discuss the interpolatory methods for second-order descriptor systems (\ref{eq:int:systemequation}). In the context of Theorem~\ref{eq:spmor-secordIndex1:thm}, dynamical systems (\ref{eq:int:systemequation}), (\ref{eq:spmor-secordIndex1:matrixvecform}), and (\ref{eq:spmor-secordIndex1:systemequation}) are equivalent. Therefore, instead of applying the proposed model reduction method onto the descriptor systems (\ref{eq:int:systemequation}), we can apply to the equivalent form (\ref{eq:spmor-secordIndex1:systemequation}).
 
\begin{theorem}
Let $G(s)= G_{1}(s)+G_2(s),$ where $G_{1}(s)$ and $G_2(s)$ 
are the strictly proper part and polynomial part, respectively, be the transfer function matrix 
of the original system and  $\hat{G}(s)= \hat{G}_{1}(s)+\hat{G}_2(s),$ where $\hat{G}_{1}(s)$ and $\hat{G}_2(s)$ are strictly proper part and polynomial part, respectively, be the transfer function matrix of its reduced system.   
If $\hat{G}(s)$ minimizes the overall error $ \| G-\hat{G}\|$, then $G_2(s)=\hat{G}_2(s)$ and $\hat{G}_{1}(s)$  minimizes the error $ \| G_1-\hat{G}_1\|$.
\end{theorem}
 
\begin{proof}
 For a proof see, e.g., \cite[Algorithm 4.1]{morGugSW13}.
\end{proof}
As a consequence of this theorem,  to apply interpolatory tangential methods via IRKA onto (\ref{eq:spmor-secordIndex1:systemequation}), the interpolation points and tangential directions are computed based on the  strictly proper part of the transfer-function matrix. One has to make sure that the reduced model has the same polynomial part as the original one. Therefore, we will modify the standard IRKA  discussed in Section \ref{sec:standard} as follows to meet  these changes.

Select a set of  interpolation points 
$\left\{\alpha_i\right\}_{i=1}^{r}$, right tangential directions
$\left\{b_i\right\}_{i=1}^{r}$ and left tangential directions 
$\left\{c_i\right\}_{i=1}^{r}$ and construct $V_s$ and $W_s$ as follows
\begin{equation}\label{eq:dae:solution2}
\begin{aligned}
V_s & = \left[(\alpha_1^2 \mathcal{M}+\alpha_1\mathcal{L}+\mathcal{K})^{-1}\mathcal{F}b_1, \cdots,(\alpha_r^2\mathcal{M}+\alpha_r\mathcal{L}+\mathcal{K})^{-1}\mathcal{F}b_r\right],\\
W_s & = \left[(\alpha_1^2 \mathcal{M}+\alpha_1\mathcal{L}+\mathcal{K})^{-T}\mathcal{H}^Tc_1, \cdots,(\alpha_r^2\mathcal{M}+\alpha_r\mathcal{L}+\mathcal{K})^{-T}\mathcal{H}^Tc_r\right].
\end{aligned}
\end{equation}

Applying $V_s$ and $W_s$ onto the system (\ref{eq:spmor-secordIndex1:systemequation}) the following reduced-order model is constructed
\begin{align}\label{eq:dae:spmor}
    \hat{\mathcal{M}}\ddot{\hat{v}}(t)+ \hat{\mathcal{L}} \dot{\hat{v}}(t)+\hat{\mathcal{K}}\hat{v}(t)  = \hat{\mathcal{F}} u(t),\quad \text{and}\quad
     \hat{y}(t) = \hat{\mathcal{H}} \hat{v}(t)+\hat{\mathcal{D}}_au(t),
\end{align} 
where the reduced matrices are formed as follows
\begin{equation}\label{eq:dae:red-matrics}
\begin{aligned}
 \hat{\mathcal{M}}=W_s^T \mathcal{M} V_s, \ \hat{\mathcal{L}}= W_s^T \mathcal{L} V_s, \\
\hat{\mathcal{K}}= W_s^T \mathcal{K} V_s, \ \hat{\mathcal{F}}=W_s^T\mathcal{F}, \ \hat{\mathcal{H}}=\mathcal{H}V_s\ \text{and}\ \hat{\mathcal{D}}_a:=\mathcal{D}_a. 
\end{aligned}
\end{equation}

These reduced matrices can also be formed using the block matrices from the descriptor system (\ref{eq:int:systemequation}) as
\begin{equation} \label{eq:dae:red-matrices}
\begin{aligned}
& \hat{\mathcal{M}} := W_s^T M_{11}V_s, \,\, \hat{\mathcal{L}} := W_s^TL_{11}V_s,\,\, 
\hat{\mathcal{K}} :=  \hat{K}_{11} - \hat{K}_{12}{K_{22}^{-1}}\hat{K}_{21}, \\  
& \hat{\mathcal{F}} := \hat{F}_1 - \hat{K}_{12}{K_{22}^{-1}}F_2, \,\,
\hat{\mathcal{H}} := \hat{H}_1 - H_2{K_{22}^{-1}}\hat{K}_{21},\,
\hat{\mathcal{D}}_a := D_a + H_2{K_{22}^{-1}}F_2,
\end{aligned}
\end{equation}
where
$\hat{K}_{11} = W_s^T K_{11}V_s,~\hat{K}_{12} = W_s^T K_{12},~\hat{K}_{21} = K_{21}V_s,~\hat{F}_1 = W_s^T F_1,~ \hat{H}_1 = H_1V_s$, 
which however show  that the reduced-matrices can be constructed  without forming the dense system (\ref{eq:spmor-secordIndex1:systemequation}). 

Now the important issue is that how to construct the transformation matrices $V_s$ and $W_s$ from the sparse system. To construct $V_s$ in (\ref{eq:dae:solution2}) at $i-$th iteration the vector $v_i= (\alpha_i^2 \mathcal{M}+\alpha_i\mathcal{L}+\mathcal{K})^{-1}\mathcal{F}b_1$ is obtained by solving the linear system 
\begin{align}\label{eq:daes:danselinearsystm}
(\alpha_i^2 \mathcal{M}+\alpha_i\mathcal{L}+\mathcal{K})v_i= \mathcal{F}b_i.
\end{align}

Plugging $\mathcal{M}$, $\mathcal{L}$, $\mathcal{K}$ and $\mathcal{F}$ from  (\ref{eq:spmor-secordIndex1:schurcomp}) we obtain
$$
(\alpha_i^2 M_{11}+\alpha_iL_{11}+K_{11}-K_{12}{K_{22}}^{-1}K_{21})v_i= (F_1 - K_{12}{K_{22}}^{-1}F_2)b_i,
$$
which implies to 
\begin{align}\label{eq:daes:effcls1}
\begin{bmatrix}\alpha_i^2  M_{11} + \alpha_i L_{11} + K_{11} & K_{12} \\ K_{21} & K_{22} \end{bmatrix}  \begin{bmatrix} v_i \\ \Gamma \end{bmatrix}  =   \begin{bmatrix} F_1 \\ F_2 \end{bmatrix} b_i,
\end{align}
for $v_i$, where $\Gamma$ is the truncated term. Although the dimension of this linear system  is higher than that of (\ref{eq:daes:danselinearsystm}), 
it is sparse and therefore, it can be treated using a sparse direct solver \cite{davis2006direct,davis2016survey}, or any suitable iterative solver \cite{Saa03a,ahamed2012iterative}
efficiently. Similarly, each vector $w_i= (\alpha_i^2 \mathcal{M}+\alpha_i\mathcal{L}+\mathcal{K})^{-T}\mathcal{H}^T c_i$ in $W_s$ of (\ref{eq:dae:solution2}) can be formed by solving the sparse linear system. Which again implies to 
\begin{align}\label{eq:daes:effcls2}
\begin{bmatrix}\alpha_i^2  {M_{11}}^T+ \alpha_i {L_{11}}^T+ K_{11}^T & K_{21}^T \\ K_{12}^T & K_{22}^T \end{bmatrix}  \begin{bmatrix} w_i \\ \Gamma \end{bmatrix}  =   \begin{bmatrix} H_1^T \\ H_2^T \end{bmatrix} c_i.
\end{align}

In this way $V_s$ and $W_s$ can be constructed without forming the dense system (\ref{eq:spmor-secordIndex1:schurcomp}) explicitly.

\subsection{Update interpolation points and tangential directions} \label{sec:updated}
We have mentioned earlier that in the tangential interpolatory methods, the selection of interpolation points, and tangential directions is an important task. Since they depend on the reduced-order model, they are not known \emph{a priori}. From Section~\ref{sec:standard} we have already known that IRKA has overcome this problem. Here we also follow Step-5 in Algorithm~\ref{alg:irka1} to select $r$ interpolation points along with left and right tangential directions. 
In our case, we construct 
\begin{align}\label{eq:daes:1storderreducedmodel}
  \hat{\mathcal{E}}:= \begin{bmatrix}
   0 & \hat{\mathcal{M}}\\
   \hat{\mathcal{M}} & \hat{\mathcal{L}}
 \end{bmatrix}, \quad \hat{\mathcal{A}}:= \begin{bmatrix}
      \hat{\mathcal{M}} &  0\\
       0 &  \hat{\mathcal{-K}}
    \end{bmatrix}, \quad \hat{\mathcal{B}}:= \begin{bmatrix}
        0\\ \hat{\mathcal{F}} 
    \end{bmatrix}\quad \text{and}\quad
    \hat{\mathcal{C}}:= \begin{bmatrix}
        0\ \hat{\mathcal{H}} 
    \end{bmatrix}.
\end{align}

Then apply Algorithm~\ref{alg:irka1} using the inputs: $\hat{\mathcal{E}}, \hat{\mathcal{A}}, \hat{\mathcal{B}}
\quad \text{and}\quad \hat{\mathcal{C}}$ to find $r\times r$ matrices $\hat{A}$ and $\hat{E}$.
The interpolation points are updated by choosing  the mirror images of the eigenvalues of the pair $(\hat{A},\hat{E})$ 
as the next interpolation points. The tangential directions are also updated similarly to Algorithm~\ref{alg:irka1}.  

The whole procedure discussed above to construct a structure-preserving reduced-order model for the second-order index-1 DAEs (\ref{eq:int:systemequation}) that summarized in Algorithm~\ref{alg:daes:irka2}.
\begin{algorithm}[t]
	\SetAlgoLined
	\SetKwInOut{Input}{Input}
	\SetKwInOut{Output}{Output}
	\caption{IRKA for Second-Order Index-1 Descriptor Systems.}
	\label{alg:daes:irka2}
	\Input {$M_{11}, L_{11}, K_{11}, K_{12}, K_{21}, K_{22}, F_1, F_2, H_1, H_2$ and $D_a$.}
	\Output  {$\hat{\mathcal{M}}, \hat{\mathcal{L}}, \hat{\mathcal{K}}, \hat{\mathcal{F}}, \hat{\mathcal{H}}$ and $\hat{\mathcal{D}}_a: = D_a + H_2{K_{22}}^{-1}F_2$ }
	Make the initial selection of the interpolation points $\{\alpha_i\}_{i=1}^r$ and the tangential directions $\{b_i\}_{i=1}^r$ and $\{c_i\}_{i=1}^r$.\\
	Construct 
	$$V_s = \left[v_1, v_2, \cdots, v_r\right] \quad \text{and}
	\quad W_s = \left[w_1, w_2, \cdots, w_r\right],$$ 
	where $v_i$ and $w_i$; $i=1,\cdots,r$ are the solutions of the linear systems (\ref{eq:daes:effcls1})
	and (\ref{eq:daes:effcls2}), respectively. \\
	\While{(not converged)}{%
	Compute $\hat{\mathcal{M}}, \ \hat{\mathcal{L}}, \ \hat{\mathcal{K}}, \ \hat{\mathcal{F}}$ and $\hat{\mathcal{H}}$ by (\ref{eq:dae:red-matrices}).\\
	Construct $\hat{\mathcal{E}}$, $\hat{\mathcal{A}}$, $\hat{\mathcal{B}}$ and $\hat{\mathcal{C}}$ as in (\ref{eq:daes:1storderreducedmodel}), then using as inputs in Algorithm~\ref{alg:irka1} and compute $\hat{A}, \hat{E}\in \mathbb{R}^{r \times r}$. \\ 
	Compute $\hat{A}z_i = \lambda_i\hat{E}z_i$ and $y^*_i \hat{A} = \lambda_i y^*_i \hat{E}$ for $\alpha_i \leftarrow -\lambda_i$, $b^*_i \leftarrow -y^*_i \hat{B}$ and $c^*_i \leftarrow \hat{C}z^*_i$, for $i=1,\cdots ,r$. \\
	Repeat Step~2. \\	
}
	Construct the reduced-order matrices
	$\hat{\mathcal{M}}, \ \hat{\mathcal{L}}, \ \hat{\mathcal{K}}, \ \hat{\mathcal{F}}$ and $\hat{\mathcal{H}}$ as in (\ref{eq:dae:red-matrices}).
\end{algorithm}

As the interpolatoty projection-based technique IRKA does not depend on the stability of the target system, the Algorithm~\ref{alg:daes:irka2} is stable and can be applied for unstable systems as well.

\subsection{Back to index-1 form} \label{backtoindex1}
 Algorithm~\ref{alg:daes:irka2} yields a standard reduced-order model (\ref{eq:dae:spmor}) from the second-order index-1 DAEs (\ref{eq:int:systemequation}). A little algebraic manipulation again turns back (\ref{eq:dae:spmor}) into an index-1 form 
\begin{subequations}
\label{eq:spmor-secordIndex1:reducedindex1form}
\begin{align}
\begin{bmatrix} \hat{\mathcal{M}} & 0 \\0 & 0 \end{bmatrix} \begin{bmatrix} \ddot{\hat{v}} (t) \\ \ddot \eta (t) \end{bmatrix} + 
\begin{bmatrix} \hat{\mathcal{L}} & 0 \\ 0 & 0 \end{bmatrix} \begin{bmatrix}\dot{\hat{v}} (t) \\ \dot \eta (t) \end{bmatrix} + 
\begin{bmatrix} \hat{K}_{11} & \hat{K}_{12} \\ \hat{K}_{21} & K_{22} \end{bmatrix} \begin{bmatrix} v (t) \\ \eta (t) \end{bmatrix}  = 
\begin{bmatrix} \hat{F}_1 \\ F_2 \end{bmatrix} u(t),\label{eq:spmor-secordIndex1:reducedindex1form1}\\
y(t) = \begin{bmatrix} \hat{H}_1 & H_2 \end{bmatrix} \begin{bmatrix} v (t) \\ \eta (t) \end{bmatrix}+D_au(t),\label{eq:spmor-secordIndex1:reducedindex1form2}
\end{align}
\end{subequations}
where all the block matrices have been defined in (\ref{eq:dae:red-matrices}). Note that this turnover, however, is not too much beneficiary if the algebraic part of the system is still large. 

\subsection{Setting with symmetric system} \label{symmetric}
When the system (\ref{eq:int:systemequation}), as defined in Section~\ref{eq:int:systemequation}, is symmetric, then the computing $V_s$ and $W_s$ in Algorithm~\ref{alg:daes:irka2} have coincided. Therefore, we can compute only Vs, and the reduced-order model in (\ref{eq:dae:spmor}), can be constructed by forming the reduced matrices in (\ref{eq:dae:red-matrices}) by using $W_s=V_s$. In this way, the constructed reduced-order model be symmetric also.  Moreover, the ROM preserves the definiteness of the original system, and the stability remains conserved.

\newlength\figwidth
\setlength{\figwidth}{.33\linewidth}
\newlength\figheight
\setlength{\figheight}{.2\linewidth}
\tikzset{mark options={solid,mark size=3,line width=5pt,mark repeat=10},line width=5pt}
\pgfplotsset{every axis plot/.append style={line width=1.5pt}}

\section{Numerical results}
In this section, we illustrate numerical results to assess the accuracy and efficiency of our proposed techniques. The techniques have applied to a set of data for the finite element discretization of Adaptive Spindle Support(ASS), \cite{kranz2009} that has already been described in Section~\ref{sec:introduction}.  In experimental data, the block matrices $M_{11}$, $L_{11}$, $K_{11}$, and $K_{22}$ are symmetric, $K_{21}=K_{12}^T$, and also the output matrix $H$ is equal to the transpose of the input matrix $F$. Therefore, the system is symmetric and hence can  exploit the symmetric properties as discussed in subsection~\ref{symmetric}. The dimension of the original model is $n=290\,137$, which consists of $n_1= 282\,699$ differential equations and $n_2=7\,438$ algebraic equations. Moreover, the number of inputs and outputs of the system is 9.

In this paper, all the results have been obtained  using MATLAB 9.5.0 (R2018b) on a Linux operating system having 24$\times$ AMD Ryzen Threadripper 1920X 12-core processor with  2.07-GHz clock speed, 128-GB of total RAM. 

\subsection{Frequency domain analysis}
We have computed the ROMs of different dimensions for the ASS model by applying  Algorithm~\ref{alg:daes:irka2}. Since the properties of the ROMs remain identical for a large number of iterations after $20$ iterations, we have continued for $20$ iterations (or cycles) at most with the tolerance  $10^{-3}$. The frequency-domain comparisons of the full model and different dimensional ROMs are demonstrating in Figure~\ref{fig:roms} on the range [$10^1 - 10^4$] [rad/s]. At each iteration of Algorithm~\ref{alg:daes:irka2}, to update the interpolation points and tangential directions, we have used Algorithm~\ref{alg:irka1} with the tolerance $10^{-5}$ and a maximum number of $20$ iterations.

Figure~\ref{fig:sigmaplot} shows the frequency responses of different dimensional reduced-order models with the full model. They are with good matching. In Figure~\ref{fig:relatverr}, the relative error between the frequency responses of the full model and reduced models have shown with upstanding accuracy.  From this figures, it can observe that the error is getting  higher if the dimension of the reduced-order model is gradually decreasing. But all the ROMs preserve the fundamental and vital attributes of the full model. Therefore, the achieved ROMs can implemented instead of the original model to perform the necessary operations of the real controller. 

\begin{figure}[!htbp]
\setlength{\figwidth}{.75\linewidth}
\setlength{\figheight}{.6\linewidth}
\begin{subfigure}{\linewidth}
\centering
\input{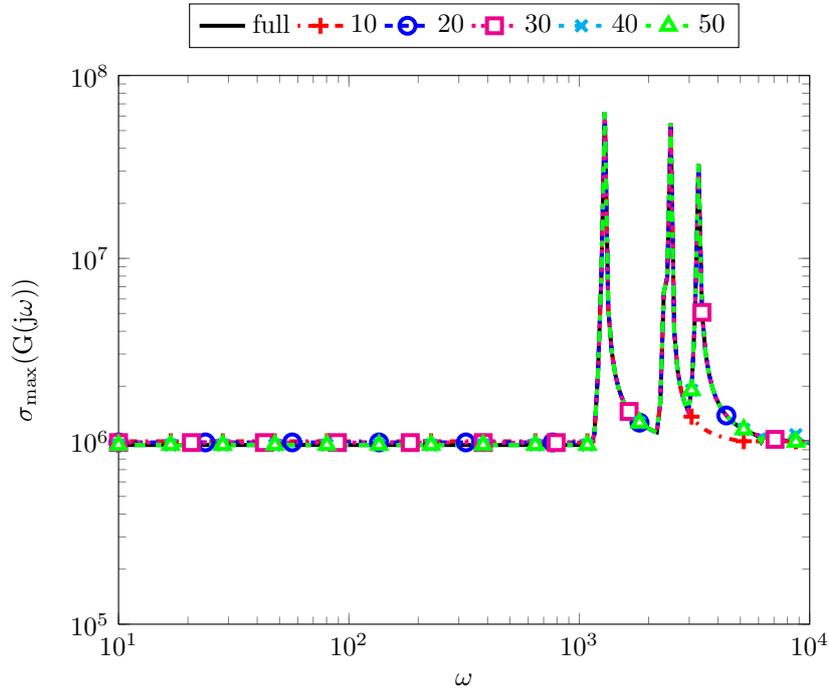}
\caption{Sigma plot}
\label{fig:sigmaplot}
\end{subfigure}
\begin{subfigure}{\linewidth}
\centering
\input{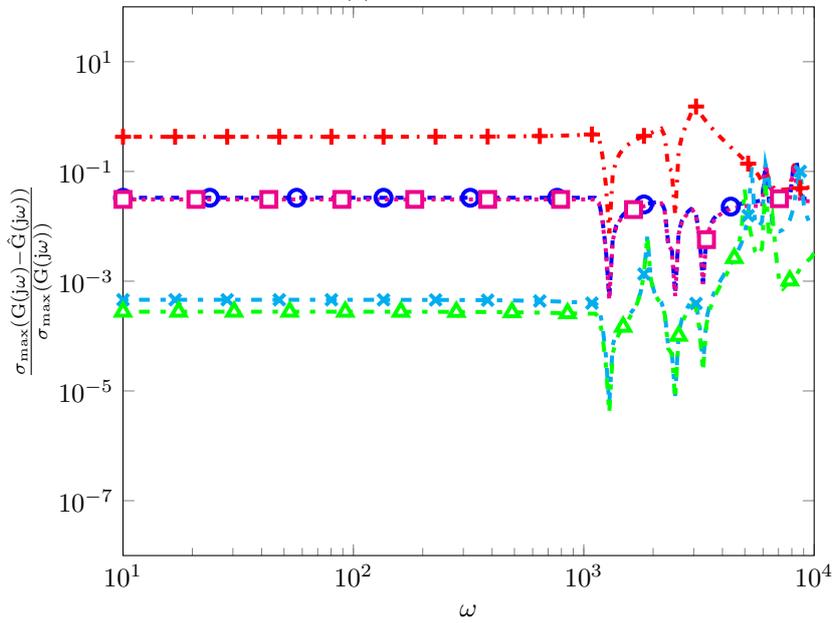}
\caption{Relative error}
\label{fig:relatverr}
\end{subfigure}
\caption{Comparisons of full and different dimensional ROMs (dimensions indicated in the legend) computed by Algorithm~\ref{alg:daes:irka2}.}
\label{fig:roms}
\end{figure}
\begin{figure}[tb]
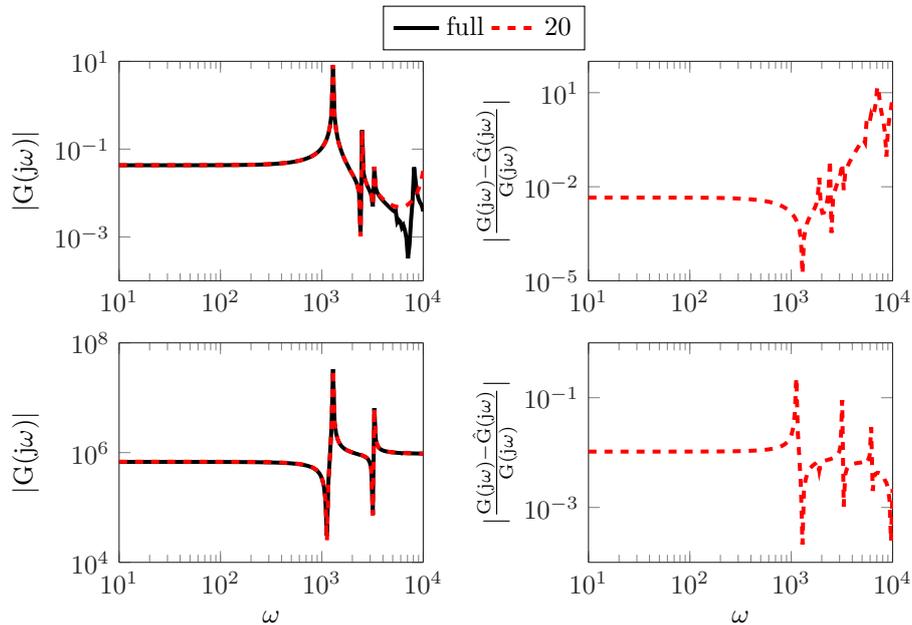

 \centering
 \begin{subfigure}[t]{.49\linewidth}
 \input{figures/tf1to9}
 \end{subfigure}
 \begin{subfigure}[t]{.49\linewidth}
%
%
%
%
\begin{tikzpicture}

\begin{loglogaxis}[%
width=\figwidth,
height=1.20\figheight,
scale only axis,
separate axis lines,
every outer x axis line/.append style={darkgray!60!black},
every x tick label/.append style={font=\color{darkgray!60!black}},
xmin=10,
xmax=10000,
every outer y axis line/.append style={darkgray!60!black},
every y tick label/.append style={font=\color{darkgray!60!black}},
ylabel=$|\frac{\text{G(j}\omega\text{)}-\hat{\text{G}}\text{(j}\omega\text{)}}{\text{G(j}\omega\text{)}}|$,
ymin=1e-05,
ymax=100
]
\addplot [
color=red,
dashed,
]
table[row sep=crcr]{
10 0.0045282622825782\\
10.3532184329566 0.0045282366427292\\
10.7189131920513 0.00452820914023913\\
11.0975249641207 0.00452817965930679\\
11.4895100018731 0.00452814808216594\\
11.8953406737032 0.00452811420649602\\
12.3155060329283 0.00452807792526348\\
12.7505124071301 0.00452803902524439\\
13.2008840083142 0.00452799731228598\\
13.6671635646201 0.00452795260973525\\
14.1499129743458 0.00452790471052101\\
14.6497139830729 0.00452785334284555\\
15.1671688847092 0.00452779831396367\\
15.7029012472938 0.00452773931152992\\
16.2575566644379 0.00452767604371587\\
16.8318035333096 0.00452760825277727\\
17.4263338600965 0.00452753559048461\\
18.0418640939207 0.00452745769574865\\
18.6791359902078 0.0045273741998971\\
19.3389175045523 0.0045272847139674\\
20.0220037181558 0.00452718879119741\\
20.7292177959537 0.00452708595971196\\
21.461411978584 0.00452697576868431\\
22.2194686093952 0.0045268576136587\\
23.0043011977292 0.00452673100318012\\
23.8168555197616 0.00452659529585221\\
24.658110758226 0.00452644979965756\\
25.5290806823952 0.00452629385458569\\
26.4308148697411 0.00452612669285537\\
27.3643999707467 0.00452594753028147\\
28.3309610183932 0.0045257555022271\\
29.3316627839005 0.00452554964220745\\
30.3677111803546 0.00452532901029311\\
31.440354715915 0.00452509253787165\\
32.5508859983506 0.00452483903658918\\
33.7006432927193 0.00452456731817233\\
34.8910121340677 0.004524276092896\\
36.1234269970943 0.00452396392842121\\
37.399373024788 0.00452362932150859\\
38.7203878181256 0.00452327067562549\\
40.0880632889846 0.00452288624587451\\
41.5040475785048 0.0045224742009885\\
42.9700470432084 0.00452203257470656\\
44.4878283112759 0.0045215592060341\\
46.0592204114511 0.00452105180844633\\
47.6861169771447 0.00452050795697567\\
49.37047852839 0.0045199250548784\\
51.1143348344017 0.00451930025267752\\
52.9197873595844 0.00451863059424426\\
54.7890117959394 0.00451791280212057\\
56.7242606849198 0.00451714348085296\\
58.7278661318948 0.00451631888540221\\
60.8022426164942 0.00451543507570453\\
62.9498899022189 0.00451448781985808\\
65.1733960488242 0.00451347250516774\\
67.4754405311069 0.00451238433018633\\
69.8587974678525 0.00451121801225245\\
72.3263389648353 0.00450996797847382\\
74.8810385759002 0.0045086282216158\\
77.5259748862946 0.00450719230288915\\
80.2643352225717 0.00450565332325118\\
83.099419493534 0.00450400393822986\\
86.0346441668451 0.00450223622852447\\
89.0735463861044 0.00450034170629586\\
92.2197882333433 0.00449831131137815\\
95.4771611420806 0.00449613532255273\\
98.8495904662559 0.00449380331836892\\
102.341140210545 0.00449130417582877\\
105.956017927762 0.00448862591069631\\
109.698579789238 0.00448575575326992\\
113.573335834311 0.00448268000477823\\
117.584955405216 0.004479383992226\\
121.738272773966 0.00447585202560594\\
126.038292967973 0.00447206728973353\\
130.49019780144 0.00446801173723625\\
135.099352119803 0.00446366610059905\\
139.871310264724 0.0044590098222527\\
144.811822767453 0.00445402073577031\\
149.926843278605 0.00444867523777854\\
155.222535742705 0.00444294808567672\\
160.705281826164 0.00443681216661009\\
166.381688607613 0.00443023855547391\\
172.258596539879 0.00442319633122527\\
178.343087693191 0.00441565236382056\\
184.642494289554 0.00440757128523378\\
191.16440753857 0.00439891519407168\\
197.916686785356 0.00438964372800353\\
204.907468981585 0.00437971362598606\\
212.145178491063 0.00436907869058247\\
219.638537241655 0.00435768968206347\\
227.396575235793 0.00434549388275099\\
235.428641432242 0.00433243509358116\\
243.744415012222 0.00431845330679383\\
252.353917043477 0.00430348455695829\\
261.267522556333 0.00428746056217672\\
270.495973046314 0.0042703086053018\\
280.050389418363 0.00425195116717836\\
289.942285388288 0.00423230567263284\\
300.183581357559 0.00421128426810901\\
310.786618778201 0.00418879342405087\\
321.764175025074 0.00416473372948928\\
333.129478793467 0.00413899956880171\\
344.896226040576 0.00411147879496621\\
357.078596490046 0.00408205239324598\\
369.691270719502 0.00405059419257856\\
382.749447851632 0.00401697059546681\\
396.268863870148 0.00398104024746228\\
410.265810582719 0.00394265375970261\\
424.75715525369 0.00390165341878878\\
439.760360930272 0.00385787303291824\\
455.293507486695 0.00381113763234588\\
471.375313411672 0.00376126336773192\\
488.025158365443 0.00370805750817203\\
505.263106533568 0.00365131823679411\\
523.109930805626 0.00359083486304129\\
541.587137807948 0.00352638803977926\\
560.716993820546 0.0034577499360986\\
580.52255160949 0.00338468488095103\\
601.027678207038 0.00330694994433911\\
622.257083673023 0.00322429587249522\\
644.236350872137 0.00313646822975525\\
666.991966303012 0.0030432089894396\\
690.551352016233 0.00294425833834067\\
714.942898659758 0.00283935698576611\\
740.195999691564 0.00272824902807906\\
766.341086800746 0.00261068530521875\\
793.409666579749 0.00248642748053112\\
821.434358491942 0.00235525293571562\\
850.448934180268 0.00221696047849178\\
880.488358164346 0.00207137696644379\\
911.588829975082 0.00191836520342357\\
943.787827777538 0.00175783288169171\\
977.12415353465 0.00158974268824025\\
1011.63797976621 0.00141412397555593\\
1047.37089795945 0.00123108551944012\\
1084.36596868961 0.0010408297513208\\
1122.66777351081 0.000843669001201987\\
1162.32246867985 0.000640048234389517\\
1203.37784077759 0.000430604817022865\\
1245.88336429501 0.000216733081605919\\
1289.89026125331 1.62620558906818e-05\\
1335.4515629299 0.000229248765773876\\
1382.62217376466 0.00045524327396867\\
1431.45893752348 0.000684239978982635\\
1482.02070579886 0.000916619316029056\\
1534.36840893001 0.00115510551222617\\
1588.56512942805 0.00140759456415523\\
1644.67617799466 0.00169493180127941\\
1702.7691722259 0.00207558310626778\\
1762.91411809595 0.00275150151161142\\
1825.18349431904 0.00492754180690727\\
1889.65233969121 0.0192540661137761\\
1956.39834351706 0.002481082952383\\
2025.50193923067 0.0014640921903613\\
2097.04640132323 0.00149687010832924\\
2171.1179456945 0.00204510293514259\\
2247.80583354873 0.00319453834615374\\
2327.20247896041 0.00585563775445266\\
2409.40356023952 0.0729533319305611\\
2494.50813523032 0.000339767617350466\\
2582.61876068267 0.00390081330479031\\
2673.84161583995 0.00677904354853994\\
2768.28663039207 0.00991995859909727\\
2866.06761694825 0.0137589469240885\\
2967.30240818887 0.018862484367478\\
3072.11299886176 0.0269425149529154\\
3180.62569279412 0.0556693935343381\\
3292.97125509715 0.00791667879541413\\
3409.28506974681 0.0276900309009565\\
3529.70730273065 0.0392520950364896\\
3654.38307095726 0.0507531065868624\\
3783.46261713193 0.0635928807910984\\
3917.10149080926 0.0784458163443771\\
4055.46073584083 0.095860549552668\\
4198.70708444391 0.116393494244396\\
4347.01315812503 0.140647184336815\\
4500.5576757005 0.169259136593935\\
4659.52566866468 0.202801956264273\\
4824.10870416537 0.241410573560481\\
4994.50511585514 0.283350418335565\\
5170.92024289676 0.317928570862546\\
5353.56667741072 0.257804509824013\\
5542.66452066311 1.2585172200304\\
5738.44164830239 0.9755299985333\\
5941.13398496503 1.17999960994392\\
6150.9857885805 2.09220060503712\\
6368.24994471859 1.65388248512702\\
6593.18827133355 2.54652458326194\\
6826.07183427239 4.5390392413085\\
7067.18127392749 15.9197391425376\\
7316.8071434272 9.9985514547299\\
7575.25025877191 3.68590888133805\\
7842.82206133768 2.04240990358329\\
8119.84499318401 1.18206279119436\\
8406.65288561832 0.542432519547715\\
8703.59136148517 0.0926126932148971\\
9011.01825166502 0.722577594417797\\
9329.30402628469 1.62089883234746\\
9658.8322411587 3.07752414207599\\
10000 7.4666302583655\\
};
\end{loglogaxis}
\end{tikzpicture}%
 \end{subfigure}
 \begin{subfigure}[t]{.49\linewidth}
 \input{figures/tf9to9}
 \end{subfigure}
 \begin{subfigure}[t]{.49\linewidth}
%
%
%
%
\begin{tikzpicture}

\begin{loglogaxis}[%
width=\figwidth,
height=1.20\figheight,
scale only axis,
separate axis lines,
every outer x axis line/.append style={darkgray!60!black},
every x tick label/.append style={font=\color{darkgray!60!black}},
xmin=10,
xmax=10000,
xlabel={$\omega$},
every outer y axis line/.append style={darkgray!60!black},
every y tick label/.append style={font=\color{darkgray!60!black}},
ylabel=$|\frac{\text{G(j}\omega\text{)}-\hat{\text{G}}\text{(j}\omega\text{)}}{\text{G(j}\omega\text{)}}|$,
ymin=0.0001,
ymax=1
]
\addplot [
color=red,
dashed,
]
table[row sep=crcr]{
10 0.0104229408844612\\
10.3532184329566 0.0104229557210426\\
10.7189131920513 0.0104229716266614\\
11.0975249641207 0.0104229886761118\\
11.4895100018731 0.0104230069487449\\
11.8953406737032 0.0104230265372778\\
12.3155060329283 0.0104230475344855\\
12.7505124071301 0.0104230700443674\\
13.2008840083142 0.0104230941677972\\
13.6671635646201 0.0104231200284193\\
14.1499129743458 0.0104231477444076\\
14.6497139830729 0.0104231774616162\\
15.1671688847092 0.0104232093072012\\
15.7029012472938 0.0104232434510738\\
16.2575566644379 0.0104232800479583\\
16.8318035333096 0.0104233192766433\\
17.4263338600965 0.0104233613267737\\
18.0418640939207 0.0104234063979995\\
18.6791359902078 0.0104234547179932\\
19.3389175045523 0.0104235065124643\\
20.0220037181558 0.0104235620250229\\
20.7292177959537 0.0104236215375952\\
21.461411978584 0.0104236853307835\\
22.2194686093952 0.0104237537112124\\
23.0043011977292 0.01042382701448\\
23.8168555197616 0.0104239055909825\\
24.658110758226 0.0104239898248123\\
25.5290806823952 0.0104240801154715\\
26.4308148697411 0.010424176909426\\
27.3643999707467 0.0104242806635531\\
28.3309610183932 0.0104243918908959\\
29.3316627839005 0.0104245111256304\\
30.3677111803546 0.0104246389419258\\
31.440354715915 0.0104247759623639\\
32.5508859983506 0.0104249228467438\\
33.7006432927193 0.0104250803108467\\
34.8910121340677 0.0104252491166629\\
36.1234269970943 0.0104254300788933\\
37.399373024788 0.0104256240830012\\
38.7203878181256 0.0104258320686585\\
40.0880632889846 0.010426055036564\\
41.5040475785048 0.0104262940799719\\
42.9700470432084 0.0104265503565653\\
44.4878283112759 0.0104268251137988\\
46.0592204114511 0.0104271196907128\\
47.6861169771447 0.0104274355189672\\
49.37047852839 0.0104277741290523\\
51.1143348344017 0.0104281371901248\\
52.9197873595844 0.0104285264577508\\
54.7890117959394 0.0104289438444329\\
56.7242606849198 0.0104293913778789\\
58.7278661318948 0.0104298712593521\\
60.8022426164942 0.0104303858391414\\
62.9498899022189 0.0104309376354264\\
65.1733960488242 0.0104315293628566\\
67.4754405311069 0.0104321639243141\\
69.8587974678525 0.0104328444512903\\
72.3263389648353 0.0104335742923065\\
74.8810385759002 0.0104343570543685\\
77.5259748862946 0.0104351966156921\\
80.2643352225717 0.0104360971277457\\
83.099419493534 0.0104370630694964\\
86.0346441668451 0.010438099244096\\
89.0735463861044 0.0104392108281386\\
92.2197882333433 0.010440403363625\\
95.4771611420806 0.010441682843688\\
98.8495904662559 0.0104430556917688\\
102.341140210545 0.010444528835748\\
105.956017927762 0.0104461097207817\\
109.698579789238 0.0104478063722337\\
113.573335834311 0.0104496274389747\\
117.584955405216 0.0104515822256937\\
121.738272773966 0.0104536807717818\\
126.038292967973 0.0104559338997857\\
130.49019780144 0.0104583532993676\\
135.099352119803 0.0104609515572816\\
139.871310264724 0.010463742296622\\
144.811822767453 0.0104667402241701\\
149.926843278605 0.0104699612308894\\
155.222535742705 0.0104734225107888\\
160.705281826164 0.0104771426828879\\
166.381688607613 0.0104811419069503\\
172.258596539879 0.0104854420144196\\
178.343087693191 0.0104900667205257\\
184.642494289554 0.0104950417583623\\
191.16440753857 0.0105003950838682\\
197.916686785356 0.0105061571041642\\
204.907468981585 0.0105123609383582\\
212.145178491063 0.0105190426737282\\
219.638537241655 0.0105262416939746\\
227.396575235793 0.0105340010314489\\
235.428641432242 0.0105423677683127\\
243.744415012222 0.0105513934874384\\
252.353917043477 0.010561134791399\\
261.267522556333 0.0105716538952202\\
270.495973046314 0.0105830192912553\\
280.050389418363 0.0105953065328887\\
289.942285388288 0.0106085991320729\\
300.183581357559 0.0106229895799478\\
310.786618778201 0.0106385805559177\\
321.764175025074 0.0106554863175951\\
333.129478793467 0.0106738343555479\\
344.896226040576 0.0106937672742637\\
357.078596490046 0.0107154451441597\\
369.691270719502 0.0107390481153206\\
382.749447851632 0.0107647797044958\\
396.268863870148 0.0107928706215311\\
410.265810582719 0.0108235834604873\\
424.75715525369 0.010857218367085\\
439.760360930272 0.0108941199480995\\
455.293507486695 0.0109346858495949\\
471.375313411672 0.0109793773132798\\
488.025158365443 0.0110287324408772\\
505.263106533568 0.0110833829260984\\
523.109930805626 0.0111440753357833\\
541.587137807948 0.0112116985564809\\
560.716993820546 0.011287319497337\\
580.52255160949 0.0113722300607675\\
601.027678207038 0.011468010062712\\
622.257083673023 0.0115766122985664\\
644.236350872137 0.0117004798896039\\
666.991966303012 0.0118427109240763\\
690.551352016233 0.0120072942279191\\
714.942898659758 0.0121994551228246\\
740.195999691564 0.0124261757709265\\
766.341086800746 0.0126970027179879\\
793.409666579749 0.0130253450778796\\
821.434358491942 0.0134306507674496\\
850.448934180268 0.0139422411182384\\
880.488358164346 0.0146064936161858\\
911.588829975082 0.0155013633693567\\
943.787827777538 0.0167687871181084\\
977.12415353465 0.0186972915719382\\
1011.63797976621 0.0219771061110488\\
1047.37089795945 0.0287718731905602\\
1084.36596868961 0.0511175465113251\\
1122.66777351081 0.229931860280657\\
1162.32246867985 0.0230086204740469\\
1203.37784077759 0.00790728760230922\\
1245.88336429501 0.00237507092834955\\
1289.89026125331 0.000212744262406645\\
1335.4515629299 0.00198047559357335\\
1382.62217376466 0.00316081337042616\\
1431.45893752348 0.00400336911352304\\
1482.02070579886 0.00463844023410187\\
1534.36840893001 0.00513700263494255\\
1588.56512942805 0.00554254674346113\\
1644.67617799466 0.00588500721629334\\
1702.7691722259 0.00618992425617828\\
1762.91411809595 0.00649376684076631\\
1825.18349431904 0.00694214359652935\\
1889.65233969121 0.00470654221070942\\
1956.39834351706 0.00648967939825951\\
2025.50193923067 0.00678745782048952\\
2097.04640132323 0.00698063191544076\\
2171.1179456945 0.00714237672190604\\
2247.80583354873 0.00729387777400229\\
2327.20247896041 0.00745859706566531\\
2409.40356023952 0.00755451203608876\\
2494.50813523032 0.00772655285743034\\
2582.61876068267 0.00793717906596166\\
2673.84161583995 0.00819203150900031\\
2768.28663039207 0.00854711510425526\\
2866.06761694825 0.00911285684761923\\
2967.30240818887 0.0102204475393573\\
3072.11299886176 0.0135683056374831\\
3180.62569279412 0.0907488517396037\\
3292.97125509715 0.00104841072542587\\
3409.28506974681 0.00397605001539409\\
3529.70730273065 0.00499815706112479\\
3654.38307095726 0.00551545849833266\\
3783.46261713193 0.00582709918627679\\
3917.10149080926 0.00603466028794924\\
4055.46073584083 0.00618244071215601\\
4198.70708444391 0.0062932405628459\\
4347.01315812503 0.00638052791500653\\
4500.5576757005 0.00645354751287267\\
4659.52566866468 0.00652005967533725\\
4824.10870416537 0.00658861703698229\\
4994.50511585514 0.00667229338787556\\
5170.92024289676 0.00680107019370582\\
5353.56667741072 0.0071343002413614\\
5542.66452066311 0.00658257073821574\\
5738.44164830239 0.0073425629637419\\
5941.13398496503 0.00882148523962969\\
6150.9857885805 0.0288565796752696\\
6368.24994471859 0.00200174101888566\\
6593.18827133355 0.0036915680804308\\
6826.07183427239 0.00414453463882987\\
7067.18127392749 0.00427420925295334\\
7316.8071434272 0.0042589951967546\\
7575.25025877191 0.00415490306267997\\
7842.82206133768 0.0039876211848491\\
8119.84499318401 0.0038482811106437\\
8406.65288561832 0.00313458133713861\\
8703.59136148517 0.00280083655040331\\
9011.01825166502 0.00220851953623942\\
9329.30402628469 0.00137094401892523\\
9658.8322411587 0.000243532378256734\\
10000 0.00217381163208334\\
};
\end{loglogaxis}
\end{tikzpicture}%
 \end{subfigure} 
 \caption{The rows respectively, show the 1st input to 9th output, 9th input to 9th output (left) and the respective relative errors (right) of the full model and 20 dimensional reduced-order model by IRKA.}
 \label{fig:assmodelsingleinput}
\end{figure}
Figure~\ref{fig:assmodelsingleinput} shows a selection of single input to single output mappings comparing the full model and $20$-dimensional reduced-order model and their corresponding relative deviations. As examples, the first input (charge)  to $9^{th}$ output (force), $9^{th}$ input (potential) to $9^{th}$ output (charge) relations of full, and the $20$- dimensional reduced model in the left side and the respective relative errors between the full and reduced model have shown on the right side of the figure~ \ref{fig:assmodelsingleinput}. Details of input-output relations can be found in \cite{morUddSKetal12}.
\begin{table}[]
\begin{center}
\caption{Speed-up comparisons for ROMs against full model by IRKA}
\label{tab:speed_up}
\begin{tabular}{ccc} \hline
	Model & Time per cycle (sec) & Speed-up \\ \hline
	full model(290137) & $100.949071$ & $1$ \\ \hline
	50 dim ROM & $0.011203$ & $9011$ \\ \hline
	40 dim ROM & $0.010584$ & $9538$ \\ \hline
	30 dim ROM & $0.009338$ & $10811$ \\ \hline
	20 dim ROM & $0.008960$ & $11267$ \\ \hline
	10 dim ROM & $0.007948$ & $12701$ \\ \hline
\end{tabular}
\end{center}
\end{table}

Table~\ref{tab:speed_up} represents the speed-up of the frequency responses of ROMs obtained by Algorithm~\ref{alg:daes:irka2} against the full model. For the time-convenient comparison, we have counted the execution time for a single cycle of the frequency responses of the full model and the ROMs. It has been observed that the obtained ROMs can highly accelerate the speed-up of the system execution time.

\setlength{\figwidth}{.75\linewidth}
\setlength{\figheight}{.6\linewidth}
\subsection{Comparison with the Balanced Truncation}
To compare the performances of IRKA and Balanced Truncation, here we have applied \cite[Algorithm 2]{benner2016structure} to the ASS model. In \cite{benner2016structure}, the authors computed different dimensional reduced-order models in different balancing levels. Exemplary, we have considered only 20 dimensional ROMs in different balancing levels; velocity-velocity, position-position, velocity-position, and position-velocity to compare with the same dimensional model obtained by IRKA. Figure~\ref{fig:com} and Figure~\ref{fig:numresult:tcom_bt_irka}, respectively, compare the accuracy and computational time between the IRKA and BT. From the relative deviations of the full model and the reduced-order models, as shown in Figure~\ref{fig:com}, we can observe that the approximation errors between IRKA and BT are almost the same except the one; position-velocity/velocity-position level. Note that in the BT method, the most expensive part is the solution of Lyapunov equations to compute the Gramian factors that are the main ingredients of this method. Once the Gramian factors are in hand, desired dimensional ROMs can be achieved by the same computational cost which has already been exhibited in  Figure~\ref{fig:numresult:tcom_bt_irka}. On the contrary, IRKA is computationally more efficient as long as consider the minimum number of the dimension of ROMs and the number of cycles as well. Figure~\ref{fig:numresult:tcom_bt_irka} also shows that the computational time is significantly increasing if the dimension of ROMs achieved by IRKA gradually increases with a constant number of cycles, i.e., 20 only. However, BT shows the same computational times to compute different dimensional ROMs. 
\begin{figure}[tb]
\begin{center}
\centering
\include{figures/comp_irka_bt20}
\caption{Relative error of 20 dimensional ROMs by the IRKA and BT}
\label{fig:com}
\end{center}
\end{figure}
\begin{figure}[tb]
\centering
\includegraphics[width=12cm,height= 8cm]{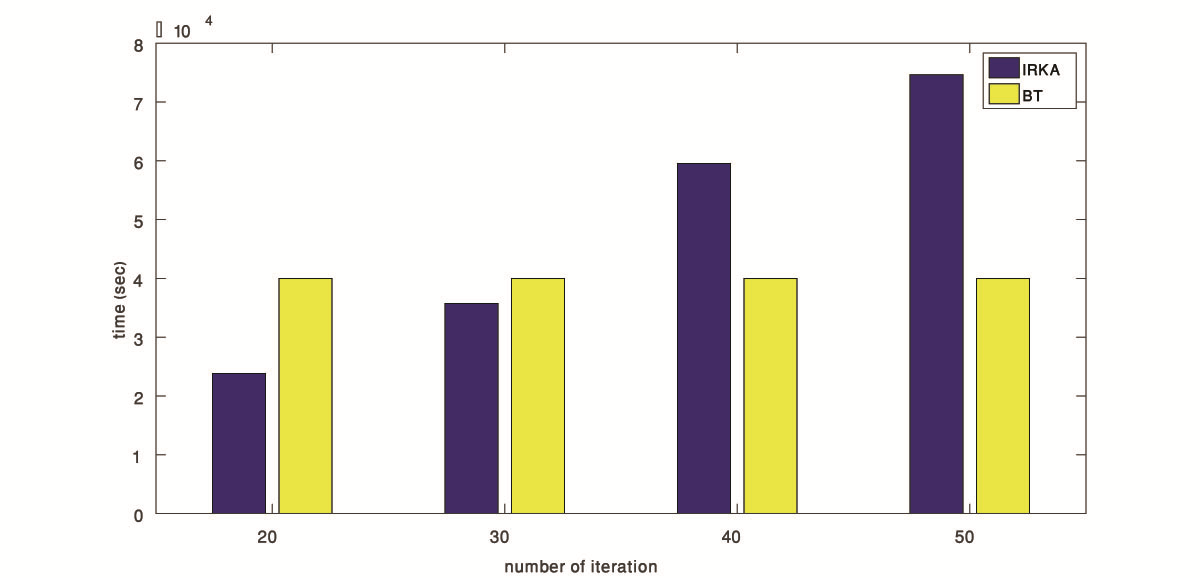}
\caption{Time comparisons for computing ROMs by IRKA and BT methods for the ASS model.}
\label{fig:numresult:tcom_bt_irka}
\end{figure}

\subsection{Stability}
Stability is one of the pivot features of a real-world system. For engineering applications, system stability is one of the fundamental requirements. In general, interpolatory projection methods do not guarantee  the stability of the ROMs. Since the ASS model is symmetric, the left and right transformation matrices, $W_s$ and $V_s$ in (\ref{eq:standard:solution2}) are the same. That is said to so a one-sided projection that guarantees the stability of the system. Figure \ref{fig:eigs} depicts the eigenvalues corresponding to all of the ROMs lie on the left-half plane in a complex domain. This figure also delineates that the successively decreasing dimensional reduced system contains the eigenvalues close to the vertical axis.
\begin{figure}[tb]
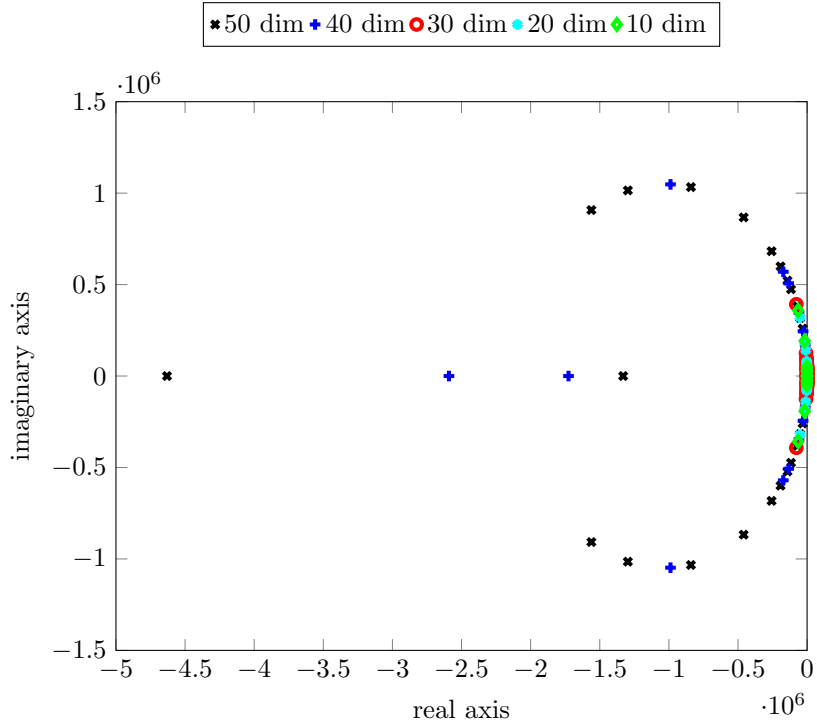

\begin{center}
\centering
\include{figures/eigs}
\caption{Eigenvalue analysis of the different dimensional ROMs computed by Algorithm~\ref{alg:daes:irka2}}
\label{fig:eigs}
\end{center}
\end{figure}

\section{Conclusions}
This paper is devoted to developing the interpolatory tangential method via IRKA for SPMOR of large-scale sparse second-order index-1 DAEs without computing the ODE system (index-0) explicitly. In this context, to modify the classical IRKA, we have discussed the techniques to construct the reduced-order matrices in sparse form by implicitly producing two transformation matrices. For this intention, the selection of interpolation points and tangential directions is a very crucial task that has been efficiently determined. We have also examined that the computational complexity can be reduced drastically for the symmetric system by constructing only one projection matrix with preserving the stability and the symmetry of the system. The performance of the proposed method has been applied to a very large model of an ASS employing piezo-actuators with 29017 DoFs, which manifests the applicability of the proposed method in real-world engineering applications. 

From the numerical computations, it has been investigated that even very lower-dimensional ROMs found by the proposed method preserve the system attributes and input-output behaviors at an acceptable level. The speed-up comparison indicates that the proposed techniques can highly accelerate the performance of the system. The transfer functions of the full model and that of the achieved ROMs are very identical in the frequency domain. Thus the achieved ROMs can be efficiently applied to maintain the production quality of the operational system and optimize the controller design to enhance the performance of the physical model. We have compared the 20 dimensional ROM achieved by IRKA with the ROMs of different levels achieved by BT. The comparison indicates the similarity of the transfer functions and behaviors of the ROMs except the velocity-position/position-velocity level attained by BT. From the time comparison of computing ROMs of different dimensions, it is evident that IRKA performs better than BT for lower-dimensions with a minimum number of iterations. Thus, IRKA provides the ROMs faster than the BT method up to a particular level. The display of the eigenvalues ensures the stability of various ROMs attained by the proposed method.

\section{Acknowledgment}
This work is partially funded by the Bangladesh Bureau of Educational Information and Statistics (BANBEIS) under the project, ID MS20191055. The first author is also a fellow of the University Grants Commission (UGC) of Bangladesh. 

\label{sec:numeric}
\bibliographystyle{IEEEtran}     
\bibliography{morbook}

\end{document}